\documentclass{amsart}
\usepackage{amsfonts,amssymb,verbatim,amsmath,amsthm,latexsym,textcomp,amscd,hyperref}
\usepackage{latexsym,amsfonts,amssymb,epsfig,verbatim}
\usepackage{amsmath,amsthm,amssymb,latexsym,graphics,textcomp}
\usepackage{graphicx}
\usepackage{color}
\usepackage{url}
\usepackage{pgfplots}
\usepackage{tikz}
\usepackage{tikz-cd}

\begin{document}

	\newtheorem{theorem}{Theorem}[section]
	\newtheorem{prop}[theorem]{Proposition}
	\newtheorem{lemma}[theorem]{Lemma}
	\newtheorem{cor}[theorem]{Corollary}
	\newtheorem{prob}[theorem]{Problem}
	\newtheorem{defn}[theorem]{Definition}
	\newtheorem{notation}[theorem]{Notation}
	\newtheorem{fact}[theorem]{Fact}
	\newtheorem{conj}[theorem]{Conjecture}
	\newtheorem{claim}[theorem]{Claim}
	\newtheorem{example}[theorem]{Example}
	\newtheorem{rem}[theorem]{Remark}
	\newtheorem{assumption}[theorem]{Assumption}
	\newtheorem{scholium}[theorem]{Scholium}

	\newcommand{\map}{\rightarrow}
	\newcommand{\C}{\mathcal C}
	\newcommand\AAA{{\mathcal A}}
	\def\AA{\mathcal A}
	
	\def\L{{\mathcal L}}
	\def\al{\alpha}
	\def\A{{\mathcal A}}

	\newcommand\GB{{\mathbb G}}
	\newcommand\BB{{\mathcal B}}
	\newcommand\DD{{\mathcal D}}
	\newcommand\EE{{\mathcal E}}
	\newcommand\FF{{\mathcal F}}
	\newcommand\GG{{\mathcal G}}
	\newcommand\HH{{\mathbb H}}
	\newcommand\II{{\mathcal I}}
	\newcommand\JJ{{\mathcal J}}
	\newcommand\KK{{\mathcal K}}
	\newcommand\LL{{\mathcal L}}
	\newcommand\MM{{\mathcal M}}
	\newcommand\NN{{\mathbb N}}
	\newcommand\OO{{\mathcal O}}
	\newcommand\PP{{\mathcal P}}
	\newcommand\QQ{{\mathbb Q}}
	\newcommand\RR{{\mathbb R}}
	\newcommand\SSS{{\mathcal S}}
	\newcommand\TT{{\mathcal T}}
	\newcommand\UU{{\mathcal U}}
	\newcommand\VV{{\mathcal V}}
	\newcommand\WW{{\mathcal W}}
	\newcommand\XX{{\mathcal X}}
	\newcommand\YY{{\mathcal Y}}
	\newcommand\ZZ{{\mathcal Z}}
	\newcommand\hhat{\widehat}
	\newcommand\flaring{{Corollary \ref{cor:super-weak flaring} }}
	\newcommand\pb{\bar{p}_B}
	\newcommand\pp{\bar{p}_{B_1}}
	\newcommand{\eg}{\overline{EG}}
	\newcommand{\eh}{\overline{EH}}
	\def\Ga{\Gamma}
	\def\Z{\mathbb Z}
	
	\def\diam{\operatorname{diam}}
	\def\dist{\operatorname{dist}}
	\def\hull{\operatorname{Hull}}
	\def\id{\operatorname{id}}
	\def\Im{\operatorname{Im}}
	
	\def\barycenter{\operatorname{center}}

	\def\length{\operatorname{length}}
	\newcommand\RED{\textcolor{red}}
	\newcommand\BLUE{\textcolor{blue}}
	\newcommand\GREEN{\textcolor{green}}
	\def\mini{\scriptsize}
	
	\def\acts{\curvearrowright}
	\def\embed{\hookrightarrow}
	
	\def\ga{\gamma}
	\newcommand\la{\lambda}
	\newcommand\eps{\epsilon}
	\def\geo{\partial_{\infty}}
	\def\bhb{\bigskip\hrule\bigskip}

	\title[Finite relative height of splitting]{ A combination theorem for relatively hyperbolic groups and finite relative height of splitting}

	\author{Ravi Tomar}
	\email{ravitomar547@gmail.com}
	\address{Department of Mathematical Sciences,
		Indian Institute of Science Education and Research Mohali,
		Knowledge City, Sector 81, SAS Nagar,
		Punjab 140306,  India}
	
	


	
	\date{\today}
	\maketitle
	\begin{abstract}
		In this paper, we prove a combination theorem for a relatively acylindrical graph of relatively hyperbolic groups (Theorem \ref{combination theorem}). Here, we are extending the technique of \cite{tomar} and constructing Bowditch boundary of the fundamental group of graph of groups. Suppose $G(\YY)$ is a graph of relatively hyperbolic groups such that edge groups are relatively quasi-convex in adjacent vertex groups. Also, assume that the fundamental group of $G(\YY)$ is relatively hyperbolic. Then we show that the edge groups of $G(\YY)$ have finite relative height (Definition \ref{rel height}) if and only if they are relatively quasi-convex (Theorem \ref{maintheorem}). In the last section, we give an application. 
	\end{abstract}
	\section{Introduction}
	In his seminal work \cite{gromov-hypgps}, Gromov introduced the notion of hyperbolic and relatively hyperbolic groups. In \cite{BF}, Bestvina-Feighn proved a celebrated combination theorem for a graph of hyperbolic groups. Using the theorem of Bestvina and Feighn, Mj and Reeves \cite{mahan-reeves} proved a combination theorem for a graph of relatively hyperbolic groups. In \cite{dahmani-comb}, with a completely different technique, Dahmani also proved a combination theorem for a graph of  relatively hyperbolic groups. Dahmani used Yaman's theorem \cite{yaman-relhyp} to prove his theorem and gave an explicit construction of Bowditch boundary, see \cite{bowditch-relhyp}. In Dahmani's theorem, one additionally requires that the action of the fundamental group of the graph of groups on its Bass-Serre tree is acylindrical and edge groups are fully quasi-convex in adjacent vertex groups. Our paper is also in the same spirit as \cite{dahmani-comb} and \cite{tomar}. Here, we are removing the hypothesis from the theorem of Dahmani that edge groups are fully quasi-convex. Also, we replace acylindrical action with a more general notion, namely relatively acylindrical action (Definition \ref{rel acyl}). In section \ref{4}, we prove the following combination theorem for relatively hyperbolic groups.
	\begin{theorem}\label{combination theorem}
		Let $G(\YY)$ be a finite graph of groups satisfying the following:
		\begin{enumerate}
			\item vertex groups are relatively hyperbolic, and edge groups are relatively quasi-convex in adjacent vertex groups,
			\item the induced parabolic structures on each edge group from adjacent vertex groups are the same,
			\item action of $G=\pi_1(G(\YY))$ on Bass-Serre tree of $G(\YY)$ is relatively acylindrical,
			\item For $n\geq k$, $\bigcap\limits_{i=1}^{n}\Lambda(G_{e_i})=\emptyset$ if $\bigcap\limits_{i=1}^{n} G_{e_i}$ is finite and $\bigcap\limits_{i=1}^{n}\Lambda(G_{e_i})$ is a singleton if $\bigcap\limits_{i=1}^{n} G_{e_i}$ is infinite, where $k$ is some fixed natural number.
		\end{enumerate}
		Then $G$ is relatively hyperbolic with respect to  subgroups corresponding to natural cone-locus. Moreover, the vertex groups are relatively quasi-convex in $G$.
	\end{theorem}
	In the above theorem, for an edge $e\in \YY$, $\Lambda(G_e)$ denotes the limit set (see Definition \ref{limit-set, rel qc}) of edge group $G_e$ in adjacent vertex groups. Also, the natural number $k$ is the constant of relative acylindricity. $\pi_1(G(\YY))$ denotes the fundamental group of the graph of groups $G(\YY)$.
	
	In Theorem \ref{combination theorem}, if we assume that edge groups are fully quasi-convex, then condition(4) is immediate as fully quasi-convex subgroups of a relatively hyperbolic group satisfy limit set intersection property (see \cite[Proposition 1.10]{dahmani-comb}). Thus, as a corollary of Theorem \ref{combination theorem}, we have the following:
	\begin{theorem}
		Let $G(\YY)$ be a finite relatively acylindrical graph of relatively hyperbolic groups such that edge groups are fully quasi-convex in adjacent vertex groups. Also, assume that $G(\YY)$ satisfies condition (2) of Theorem \ref{combination theorem}. Then the fundamental group of $G(\YY)$ is relatively hyperbolic with respect to subgroups corresponding to natural cone-locus.
	\end{theorem}
	Another special case of Theorem \ref{combination theorem} is the following:
	\begin{cor}
		Let $G(\YY)$ be an acylindrical finite graph of relatively hyperbolic groups with relatively quasi-convex edge groups. Assume that, for $n\geq k$, if $\bigcap\limits_{i=1}^{n} G_{e_i}$ is finite then $\bigcap\limits_{i=1}^{n}\Lambda(G_{e_i})=\emptyset$. Then the fundamental group $G$ of $G(\YY)$ is relatively hyperbolic with respect to the collection of images of parabolic subgroups of vertex groups in $G$.
	\end{cor}
	In \cite{ilyakapovichcomb}, I. Kapovich proved that the fundamental group of an acylindrical graph of hyperbolic groups with quasi-convex edge groups is hyperbolic, and all vertex groups are quasi-convex in the fundamental group of the graph of groups. An analogue of this result in relatively hyperbolic group setting was proved by Dahmani \cite{dahmani-comb} with the additional hypothesis that edge groups are fully quasi-convex. In this article, we are removing this additional hypothesis by introducing the notion of relatively acylindrical action. Thus, Theorem \ref{combination theorem} is an analogous version of I.Kapovich's result for relatively hyperbolic groups.
	
	In \cite{tomar}, the author proved a combination theorem for a graph of relatively hyperbolic groups with parabolic edge groups and constructed the Bowditch boundary. There, the author generalizes the technique of \cite{dahmani-comb} to construct the Bowditch boundary. In this paper, we are considering the general case when edge groups are relatively quasi-convex. Here, we are extending the technique of \cite{tomar} to construct the Bowditch boundary of the fundamental group of $G(\YY)$ (as in Theorem \ref{combination theorem}). In \cite[Theorem 1.6]{tomar}, edge groups are parabolic, so their limit sets are singletons. Thus, only one point in the edge boundary can have an infinite domain, see Definition \ref{domain}. However, in our situation, there can be more than one parabolic point in an edge boundary with infinite domains. Note that only parabolic points can have infinite domains by the third condition in Theorem \ref{combination theorem}. In the construction of Bowditch boundary, we also require that the intersection of infinite domains is uniformly finite. It can be achieved by the last two conditions in Theorem \ref{combination theorem}. To prove the above theorem, we introduced the notion of relatively acylindrical action, see Definition \ref{rel acyl}, which generalizes the notion of acylindrical action.
	\begin{rem}
		In Theorem \ref{combination theorem}, we are not assuming that the monomorphisms from edge groups to vertex groups are quasi-isometric embeddings. Thus, condition (1) of \cite[Theorem 4.5]{mahan-reeves} is not satisfied. Therefore, Theorem \ref{combination theorem} does not follow from the combination theorem of Mj-Reeves \cite{mahan-reeves}.
	\end{rem}
	
	Let $(G,\PP)$ be a relatively hyperbolic group. An element $f\in G$ is {\em elliptic} if the order of $f$ is finite. If $f$ has infinite order, then $f$ is {\em parabolic} if it is contained in a conjugate of some parabolic subgroup $P\in \PP$ and {\em loxodromic} otherwise. The notion of the height of a subgroup was introduced in \cite{GMRS}. There, they proved that quasi-convex subgroups of a hyperbolic group have finite height. This notion was generalized by Hruska and Wise \cite{hruska-wise} in the setting of relatively hyperbolic groups. In \cite{hruska-wise}, the authors proved that relatively quasi-convex subgroups of a relatively hyperbolic group have finite relative height. Now, we recall the definition of the relative height of a subgroup of a relatively hyperbolic group.
	
	\begin{defn}[Relative height\cite{hruska-wise}]\label{rel height}
		Let $(G,\PP)$ be a relatively hyperbolic group and let $H$ be a subgroup of $G$. Relative height of $H$ is $n$, if $n$ is the maximum number of distinct cosets $g_iH$ such that $\bigcap\limits_{i=1}^{n}g_iHg_i^{-1}$ contains a loxodromic element.
	\end{defn} 
	
	It is well known that quasi-convex subgroups of a hyperbolic group have finite height \cite{GMRS}. The converse of this fact is an open question. This question is open even when the subgroup is malnormal (subgroup with height $1$). In \cite{mahan-height-split}, Mitra proved that when a hyperbolic group $G$ split as an amalgamated free product or HNN-extension of hyperbolic groups with quasi-convex edge groups, edge groups have finite height in $G$ if and only if they are quasi-convex in $G$. Recently, Pal \cite{pal-rel-height} generalized the work of Mitra \cite{mitra-ht} in the setting of relatively hyperbolic groups. The following theorem generalizes the work of \cite{pal-rel-height}.
	\begin{theorem}\label{maintheorem}
		Let $G(\YY)$ be a finite graph of groups satisfying the following:
		\begin{enumerate}
			\item vertex groups are relatively hyperbolic, and edge groups are relatively quasi-convex in adjacent vertex groups,
			\item the induced parabolic structures on each edge group from adjacent vertex groups are the same,
			\item $G=\pi_1(G(\YY))$ is relatively hyperbolic with respect to subgroups corresponding to natural cone-locus,
			\item For $n\geq k$, $\bigcap\limits_{i=1}^{n}\Lambda(G_{e_i})=\emptyset$ if $\bigcap\limits_{i=1}^{n} G_{e_i}$ is finite and $\bigcap\limits_{i=1}^{n}\Lambda(G_{e_i})$ is a singleton if $\bigcap\limits_{i=1}^{n} G_{e_i}$ is infinite, where $k$ is some fixed natural number.
		\end{enumerate}
		Then, the edge groups have finite relative height in $G$ if and only if they are relatively quasi-convex in $G$.
	\end{theorem}
	In Theorem \ref{maintheorem}, if we assume that the edge groups are fully quasi-convex, then condition(4) is immediate. Thus, we immediately have the following:
	\begin{theorem}
		Let $G(\YY)$ be a finite graph of relatively hyperbolic groups with fully quasi-convex edge groups. Also, assume that $G(\YY)$ satisfies conditions (2) and (3) of Theorem \ref{maintheorem}. Then edge groups have finite relative height in $G$ if and only if they are relatively quasi-convex in $G$.
	\end{theorem}
	In \cite{pal-rel-height}, Pal assumed that the monomorphisms from edge groups to vertex groups are quasi-isometric embeddings. Also, he took a more restrictive parabolic structure (like in \cite{dahmani-comb}) on $\pi_1(G(\YY))$. In Theorem \ref{maintheorem}, we are relaxing these assumptions and taking a more general parabolic structure (like in \cite{mahan-reeves}) on $\pi_1(G(\YY))$.
	
	The following corollary of Theorem \ref{maintheorem} shows that, under a mild extra assumption, we can obtain \cite[Theorem 1.2]{pal-rel-height}. Note that we are not assuming that the monomorphisms from edge groups to vertex groups are qi embeddings.
	\begin{cor} 
		Let $G(\YY)$ be a finite graph of relatively hyperbolic groups with relatively quasi-convex edge groups. Assume that the fundamental group $G$ of $G(\YY)$ is relatively hyperbolic with respect to the collection of images of parabolic subgroups of vertex groups in $G$. Also, assume that, for $n\geq k$, if $\bigcap\limits_{i=1}^{n} G_{e_i}$ is finite then $\bigcap\limits_{i=1}^{n}\Lambda(G_{e_i})=\emptyset$ where $k$ is a fixed natural number. Then, the edge groups have finite relative height in $G$ if and only if they are relatively quasi-convex in $G$.
	\end{cor}

	{\bf Acknowledgment:} The author sincerely thanks his supervisor, Pranab Sardar, for many fruitful discussions and comments on the exposition of the paper. The author owes him for pointing out relative acylindrical condition in Theorem \ref{combination theorem}.
	\section{Preliminaries}\label{1}
	We collect here some necessary definitions from \cite{bowditch-relhyp},\cite{bowditch-cgnce},\cite{dahmani-comb}.
	
	A group $G$ is said to be a {\em convergence group} if there exists a compact metrizable space $X$ and $G$ acts on $X$ such that: given any sequence $(g_n)_{n \in \mathbb{N}}$ in $G$, there exists a subsequence $(g_{n_k})_{k \in \mathbb{N}}$ and two points $x$, $y$ in $X$ such that  for all compact subset $K$ of $X\setminus \{y\}$, $g_{n_k}K$ converges uniformly to $x$.
	
	Let $G$ be a group acting as a convergence group on a compact metrizable space $X$. An infinite order element $g \in G$ is said to be \emph{loxodromic} if it has exactly two distinct fixed points in $X$. A point $x\in X$ is said to be a {\em conical limit point} if there exists a sequence $(g_n)_{n \in \mathbb{N}}$ in $G$ and two distinct points $z$, $y$ such that $g_nx$ converges to $z$ and for all $x' \in X\setminus \{x\}$, $g_nx'$ converges to $y$. A subgroup $P$ of $G$ is said to be \emph{parabolic} if it is infinite, does not have any loxodromic element, and fixes a point $p$. Such a point is unique and called a \emph{parabolic point}. The stabilizer of a unique parabolic point is called a {\em maximal parabolic subgroup}. A parabolic point $p$ is called a bounded parabolic point if $Stab_{G}(p)$ acts co-compactly on $X\setminus\{p\}$. Group $G$ is said to be {\em geometrically finite} if every point of $X$ is either a conical limit point or a bounded parabolic point.
	\begin{rem}
		Let $H$ be a subgroup of $G$, and $G$ acts as a convergence group on a compact space $X$. Then, every conical limit point for $H$ action on $\Lambda (H) \subset X$ (see Definition \ref{limit-set, rel qc}(1) below) is a conical limit point for $H$ acting in $X$ and hence for $G$ in $X$. Each parabolic point for $H$ in $\Lambda (H)$ is a parabolic point for $G$ in $X$ and its maximal parabolic subgroup in $H$ is precisely the intersection of maximal parabolic subgroup in $G$ with $H$.
	\end{rem}
	Next, we define relatively hyperbolic groups introduced by Gromov \cite{gromov-hypgps}. Several other definitions of relative hyperbolicity are known, see \cite{hru-rel},\cite{osin-book},\cite{bowditch-relhyp}. In \cite{hru-rel}, Hruska shows that all these definitions are equivalent when group is countable, and the collection of peripheral subgroups is finite. Here, we use the dynamical definition of Bowditch \cite[Definition 1]{bowditch-relhyp}.
	\begin{defn}
		[Relatively hyperbolic group] A group $G$ is said to relatively hyperbolic with respect to a collection of  subgroups $\PP$ if $G$ admits a properly discontinuous action by isometries on a proper hyperbolic space $X$ such that the induced action on $\partial X$ (Gromov boundary of $X$) is convergence, every point of $\partial X$ is either conical limit point or bounded parabolic point, and the subgroups in $\mathbb{H}$ are precisely the maximal parabolic subgroups. 
	\end{defn}
	In short, we say that the pair $(G,\PP)$ is a relatively hyperbolic group. In this case, the boundary of $X$ is canonical and called the Bowditch boundary of $G$. We shall denote it by $\partial G$. Now, we have the following topological characterization of relatively hyperbolic groups given by Yaman \cite{yaman-relhyp}.
	\begin{theorem}
		[\cite{yaman-relhyp}]\label{yaman} Let $G$ be  a geometrically finite group acting on a non-empty perfect metrizable compactum $X$. Assume that the quotient of bounded parabolic points is finite under the action of $G$, and the corresponding maximal parabolic subgroups are finitely generated. Let $\PP$ be the family of maximal parabolic subgroups. Then $(G,\PP)$ is a relatively hyperbolic group and $X$ is equivariantly homeomorphic to Bowditch boundary of $G$. 
	\end{theorem}
	
	\begin{defn} We collect some more definitions:\label{limit-set, rel qc}
		\begin{enumerate}
			\item \textup{(Limit set \cite{dahmani-comb})} Let $G$ be a convergence group on $X$. The {\em limit set} $\Lambda(H)$ of  an infinite subgroup $H$ is the unique minimal non-empty closed $H$-invariant subset of $X$. The limit set of a finite set is empty.
			\item \textup{(Relatively quasi-convex subgroup \cite{dahmani-comb})} Let $G$ be a relatively hyperbolic group with Bowditch boundary $\partial G$. Let $H$ be a group acting as a geometrically finite convergence group on a compact metrizable space $\partial H$. We assume that $H$ is embedded in $G$ as a subgroup. $H$ is \emph{relatively quasi-convex} if $\Lambda(H) \subset \partial G$ is equivariantly homeomorphic to $\partial H$.
		\end{enumerate}
	\end{defn}
	Let $(G,\PP)$ be a relatively hyperbolic group and let $H$ be a relatively quasi-convex subgroup of $G$. Then, by \cite[Theorem 9.1]{hru-rel}, we have induced peripheral structure $\PP_H$ on $H$.
	\begin{center}
		$\PP_H=\{H\cap gPg^{-1}:g\in G, P\in \PP$, \text{and}  $H\cap gPg^{-1}$ \text{is infinite}\}
	\end{center} 
	Let $G(\YY)$ be a graph of groups as in Theorem \ref{combination theorem}. Since each edge group is relatively quasi-convex in adjacent vertex groups, we have induced parabolic structures on the edge group. We assume that, as a set, these induced parabolic structures on the edge groups are the same. Then, every parabolic subgroup in an edge group sits inside corresponding maximal parabolic subgroups in adjacent vertex groups. In this way, we obtain a subgraph of groups of $G(\YY)$ where each vertex group is maximal parabolic in the corresponding vertex group of $G(\YY)$, and each edge group is maximal parabolic in the corresponding edge group of $G(\YY)$. Note that the fundamental groups of these subgraphs of groups are subgroups of the fundamental group of $G(\YY)$. We say that these subgroups correspond to the natural cone locus. The notion of cone locus comes from Mj-Reeves \cite{mahan-reeves}, where one has a tree of relatively hyperbolic spaces with qi embedding condition such that the qi embeddings from edge spaces to vertex spaces are strictly type preserving. Then, we have cone locus corresponding to parabolics of vertex and edge spaces. 
	
	Let $G(\YY)$ be a graph of groups and let $T$ be its Bass-Serre tree. The action of $G=\pi_1(G(\YY))$ is said to be {\em acylindrical} if the stabilizer of a geodesic segment in $T$ of length greater than $k$ is finite for some $k\in\mathbb{N}$. When $G(\YY)$ is a graph of relatively hyperbolic groups with relatively quasi-convex edge groups (more generally for a graph of convergence groups with dynamically quasi-convex edge groups \cite{tomar}), we generalize acylindrical action to relative acylindrical action.
	\begin{defn}
		[Relatively acylindrical action] \label{rel acyl}Let $G(\YY)$ be a graph of relatively hyperbolic groups with relatively quasi-convex edge groups. The action of $G=\pi_1(G(\YY))$ on Bass-Serre tree $T$ of $G(\YY)$ is said to be $k$-relatively acylindrical if the stabilizer of any geodesic $\alpha$ in $T$ of length at least $k$ is a parabolic subgroup in the stabilizer of each edge in $\alpha$. If there is no parabolic in edge groups, then we say that the action of $G$ on $T$ is $k$-relatively acylindrical if it is $k$-acylindircal. The action of $G$ on $T$ is relatively acylindrical if it is $k$-relatively acylindrical for some $k\in \mathbb{N}$.
	\end{defn}
	\begin{example}
		Let $G=G_1\ast_P G_2$, where $G_1$ and $G_2$ are relatively hyperbolic and $P$ is a parabolic subgroup sitting properly in corresponding maximal parabolic subgroups of $G_1$ and $G_2$. Then the action of $G$ on its Bass-Serre tree is $2$-relatively acylindrical.
	\end{example}
	\section{Construction of Bowditch boundary}\label{2}
	Let $G(\YY)$ be the graph of groups as in Theorem \ref{combination theorem}, and let $G$ be the fundamental group of $G(\YY)$. In this section, we construct a compact metrizable space $X$ on which there is a natural action of $G$. In the next section, we show that this action is geometrically finite. Then, by invoking Theorem \ref{yaman}, $G$ is a relatively hyperbolic group.
	
	We fix some notation: Let $T$ be the Bass-Serre tree of $G(\YY)$. For a vertex $v$ of  $T$, we write the vertex stabilizer as $G_v$. Similarly, for an edge $e$ of $T$, we write the edge stabilizer as $G_e$. We write $X_v$ and $X_e$ as compact metrizable space on which $G_v$ and $G_e$ act as convergence groups.
	\subsection{Definition of $X$ as a set}\hfill
	
	\textbf{Contribution of the vertices of $T$}
	
	Let $\tau$ be a subtree of $T$ such that $G.\tau=T$. Let $\VV(\tau)$ be the set of vertices of $\tau$. For each vertex $v\in\VV(\tau)$, the group $G_v$ acts geometrically finitely on compact metrizable space $X_v$ (Bowditch boundary of $G_v$). Set $\Omega$ to be $G \times(\displaystyle\bigsqcup_{v\in \VV(\tau)}X_v)$ divided by the natural relation:
	
	\[(g_1,x_1) = (g_2,x_2) \text{ if }  \exists v\in \VV(\tau), x_i\in X_v,g_2^{-1}g_1 \in G_v, (g_2^{-1}g_1)x_1 =x_2 \text{ for } i=1,2\]
	
	In this way, $\Omega$ is the disjoint union of compactums corresponding to the  stabilizers of the vertices of $T$. Also, for each $v\in \VV(\tau)$, the space $X_v$ naturally embeds in $\Omega$ as the image of $\{1\}\times X_v$. We identify it with its image. The group $G$ naturally acts on the left on $\Omega$. For $g\in G$, $gX_v$ is the compactum for the vertex stabilizer $G_{gv}$.
	
	\textbf{Contribution of the edges of $T$} 
	
	Each edge allows us to glue together compactums corresponding to the vertex stabilizers along with the limit set of the stabilizer of the edge. Each edge group embeds as a relatively quasi-convex subgroup in adjacent vertex groups. Thus, for each edge $e\in T$, group $G_e$ is relatively hyperbolic with Bowditch boundary $X_e$. Let $e$ = $(v_1,v_2)$ be the edge in $\tau$. Then, by our assumption, there exist equivariant embeddings $\Lambda_{e,v_i}:X_e \hookrightarrow X_{v_i}$ for $i=1,2$. Similar maps are defined by translation for edges in $T\setminus \tau$.
	
	The equivalence relation $\sim$ on $\Omega$ is the transitive closure of the following: Let $v$ and $v'$ be vertices of $T$. The points $x \in X_v$ and $x' \in X_{v'}$ are equivalent in $\Omega$ if there is an edge $e$ between $v$ and $v'$ and a point $x_e \in X_e$ satisfying $x = \Lambda_{e,v}(x_e)$ and $x' = \Lambda_{e,v'}(x_e)$ simultaneously. Let $\Omega/_{\sim}$ be the quotient under this relation and let $\pi':\Omega \map \Omega/_{\sim}$  be the corresponding projection. An equivalence class $[x]$ of an element $x\in \Omega$ is denoted by $x$ itself.
	
	Let $\partial T$ be the (visual) boundary of the tree. We define $X'$ as a set: $X' = \partial T \sqcup (\Omega/_{\sim})$. 
	\begin{defn}\label{domain}
		$\bf{(Domains)}$ For all $x \in \Omega/_{\sim}$, we define the {\em domain} of $x$ to be $D(x) = \{v\in \VV(T) | x\in \pi'(X_v)\}$. We also say that the domain of a point $\xi\in \partial T$ is $\{\xi\}$ itself.
	\end{defn}
	\subsection{Final constrution of $X$}\hfill
	
	It turns out that $X'$ with the topology defined in \cite{dahmani-comb} is not a Hausdorff space. Therefore we need to modify $X'$ further. For getting the desired space, we  further define an equivalence relation on the set $X'$. Note that points coming from edge boundaries can have infinite domains. Then, by our hypotheses, the points in $X'$ whose domains are infinite are parabolics. Also, if $x_1,x_2 \in X_e$ are parabolics for some edge $e\in T$ and $D(x_1)$ and $D(x_2)$ are infinite, then $D(x_1)\cap D(x_2)$ is uniformly finite. So, we identify the boundary points of the infinite domain $D(x)$ in $\partial T$ to $x$ itself. By considering the equivalence relation generated by these relations, we denote the quotient of $X'$ by $X$. Again, $X$ can be written as a disjoint union of two sets of equivalence classes :
	\begin{center}
		$X = \Omega' \sqcup (\partial T)'$
	\end{center} 
	where $\Omega'$ (as a set, it is the same as $\Omega$) is the set of equivalence classes of elements in $\Omega$, and, $(\partial T)'$ are the equivalence classes of the remaining elements in $\partial T$ as some elements of $\partial T$ are identified with parabolic points of edge groups. The equivalence class of each remaining element in $\partial T$ is a singleton. 
	\begin{rem}
		The elements of $\Omega'$ will be denoted by $x,y,z...$ and the elements of $(\partial T)'$ will be denoted by $\xi,\eta,\zeta....$. In $X$, we define the domain of each element as previously defined. Additionally, for those points $\eta$ of $\partial T$, which are identified with parabolic  points $x$ of edge groups, we define the domain of $\eta$ same as the domain of $x$. 
	\end{rem}
	It is clear that for each $v$ in $T$, the restriction of projection map $\pi'$ from  $X_v$ to $\Omega$ is injective. Let $\pi''$ be the projection map from $X'$ to $X$. Let $\pi$ be the composition of the restriction of $\pi''$ to $\Omega$ and $\pi'$. Again the restriction of $\pi$ to $X_v$ is injective for all $v$ in $T$.
	\subsection{Definition of neighborhoods in $X$}\hfill
	
	We define a family $(W_n(x))_{n\in \mathbb{N},x\in X}$ of subsets of $X$, which generates a topology on $X$. For a vertex $v$ and an open subset $U$ of $X_v$, we define the subtree $T_{v,U}$ of $T$ as $\{w\in \VV(T) : X_e\cap U\neq \emptyset\}$, where e is the first edge of $[v,w]$. For each vertex $v$ in $T$, let us choose $\UU(v)$, a countable basis of open neighborhoods of $X_v$. Without loss of generality, we can assume that for all $v$, the collection of open subsets $\UU(v)$ contains $X_v$. Let $x\in\Omega'$ and $D(x) = \{v_1,...,v_n,...\} = (v_i)_{i\in I}$. Here $I$ is a subset of $\mathbb{N}$. For each $i\in I$, let $U_i \subset X_{v_i}$ be an element of $\UU(v_i)$ containing $x$ such that for all but finitely many indices $i\in I, U_i=X_{v_i}$. The set $W_{(U_i)_i\in I}(x)$ is the disjoint union of three subsets 
	\begin{center}
		$W_{(U_i)_{i\in I}}(x) = A \cup B\cup C$,
	\end{center} 
	
	where $A$ is nothing but the collection of all boundary points of subtrees $T_{v_i,U_i}$ which are not identified with some parabolic point corresponding to an edge group, $B$ is the collection of all points $y$ outside $\bigcup_{i\in I}X_{v_i}$ in $\Omega'$ whose domains lie inside $\bigcap_{i\in I}T_{v_i,U_i}$, and $C$ contains all elements in $\bigcup_{j\in I}X_{v_j}$ which belong to $\bigcap_{m\in I | y \in X_{v_m}}U_m$. In notations $A$,$B$,$C$ are defined as follows:
	\begin{align*}
		A &= (\bigcap_{i\in I} \partial T_{v_i,U_i}) \cap (\partial T)'\\
		B &= \{y \in \Omega'\setminus(\bigcup_{i\in I}X_{v_i}) | D(y) \subset \bigcap_{i\in I}T_{v_i,U_i}\}\\
		C &=\{y \in \bigcup_{j\in I}X_{v_j} | y \in \bigcap_{m\in I | y \in X_{v_m}}U_m\}
	\end{align*}
	
	As $A \subset (\partial T)'$, the remaining elements in $\displaystyle\bigcap_{i\in I} \partial T_{v_i,U_i}$ are in $B$. In this way, $A$, $B$, $C$ are disjoint subsets of $X$. 
	\begin{rem}
		The set $W_{(U_i)_{i\in I}}(x)$ is completely defined by the data of the domain of $x$, the data of a finite subset $J$ of $I$, and the data of an element of $\UU(v_j)$ for each index $j\in J$. Therefore there are only countably many different sets  $W_{(U_i)_{i\in I}}(x)$, for $x\in \Omega'$, and $U_i\in \UU(v_i),v_i\in D(x)$. For each $x$, we choose an arbitrary order and denote them by $W_m(x)$.
	\end{rem}
	Now, we  define neighborhoods around the points of $(\partial T)'$. Let $\eta\in (\partial T)'$ and choose a base point $v_0$ in $T$. Firstly, we define the subtree $T_m(\eta)$: it consists of the vertices $w$ such that $[v_0,w] \cap [v_0,\eta)$ has length bigger than $m$. We set $W_m(\eta) = \{\zeta \in X | D(\zeta) \subset T_m(\eta)\}$. This definition does not depend on the choice of base point $v_0$, up to shifting the indices.
	\subsection{Topology of $X$}\hfill 
	
	Consider the smallest topology on $X$ such that the family of sets $\{W_n(x):x\in X, n\in\mathbb{N}\}$ are open subsets of $X$. The following lemmata show that $X$ is a compact metrizable space.
	\begin{lemma}[Avoiding an edge]\label{avoiding an edge} Let $x\in X$ and let $e$ be an edge of $T$ such that at least one vertex of $e$ is not in $D(x)$. Then, there exists $W_n(x)$ such that $X_e\cap W_n(x)=\emptyset$.
	\end{lemma}
	\begin{proof}
		Let $e$ be an edge of $T$ such that vertex $v_0$ of $e$ does not belong to $D(x)$. If $\eta\in (\partial T)'$, then by choosing $v_0$ as a base point, we see that $W_n(\eta)\cap X_e=\emptyset$ for $n\geq 1$. If $x\in X_v$ for some $v\in V(T)$, then there is a unique path $\alpha$ in $T$ joining $v$ to $v_0$. Since $x\notin D(x)$, let $e_1$ be the first edge in $\alpha$ such that $x\notin X_{e_1}$. Let $v_1$ be the vertex of $e_1$ such that $x\in X_{v_1}$ but $x\notin X_{e_1}$. To prove the lemma, it is sufficient to find a neighborhood $U_{v_1}$ of $x$ in $X_{v_1}$ such that $U_{e_1}\cap X_{e_1}=\emptyset$. Since $X_{e_1}$ is compact and $x\notin X_{e_1}$, then there exist $U_{e_1}$ such that $U_{e_1}\cap X_{e_1}=\emptyset$.
	\end{proof}
	Hausdorffness of $X$ follows from the following:
	\begin{lemma}
		The space $(X,\TT)$ is Hausdorff.
	\end{lemma}
	\begin{proof}
		Let $x$ and $y$ be two distinct points of $X$. Suppose that $D(x)\cap D(y)=\emptyset$. Then there is a unique geodesic $\alpha$ from $D(x)$ to $D(y)$. Since $D(x)\cap D(y)=\emptyset$, there exists an edge $e$ in $\alpha$ such that both vertices of $e$ are neither contained in $D(x)$ nor in $D(y)$. Then, by Lemma \ref{avoiding an edge}, we have disjoint neighborhoods $W_n(x)$ and $W_(y)$. If $D(x)\cap D(y)\neq \emptyset$. Note that $D(x)\cap D(y)$ is finite: otherwise, $x$ will be equal to $y$. Let $D(x)\cap D(y)=\{v_1,v_2,...,v_n\}$ for some $n\in\mathbb{N}$. Since $X_v$ is Hausdorff for $v\in V(T)$, choose disjoint neighborhoods $U_{v_i},V_{v_i}$ around $x,y$ respectively in $X_{v_i}$ for $i=1,2...,n$. Using the neighborhoods, we have disjoint neighborhoods $W_n(x),W_m(y)$ around $x,y$ respectively.
	\end{proof}
	\begin{lemma}[Filtration]
		For every $x$, every integer $n$, every $y\in W_n(x)$, there exists $m$ such that $W_m(y)\subseteq W_n(x)$.
	\end{lemma}
	\begin{proof}
		Let $\eta\in (\partial T)'$ and let $y\in W_n(\eta)$. If $\xi\in (\partial T)'$, then choose $n=m$ and $W_n(x)=W_m(y)$. Let $y\in \Omega'$. Then $D(y)\subseteq T_n(\eta)$. Suppose the subtree $T_n(\eta)$ starts at vertex $v$ and let $e$ be the last edge on a geodesic segment from a base vertex to $v$. Observe that points of $X_e$ do not belong to $W_n(\eta)$. Thus, there exists $W_m(y)$ such that $W_m(y)\subseteq W_n(\eta)$ for sufficiently large $m$. Now, suppose that $x\in \Omega'$ and let $y\in W_n(x)$. If $D(x)\cap D(y)=\emptyset$, then there exists an edge on a unique geodesic segment joining $D(x)$ to $D(y)$. Then, by Lemma \ref{avoiding an edge}, one can find a neighborhood $W_m(y)$ such that $W_m(y)\subseteq W_n(x)$. If $D(x)\cap D(y)\neq\emptyset$, then this intersection is finite. Let $D(x)\cap D(y)=(v_i)$, where $i=1,2,...,k$ for some integer $k$. Note that $y\in U_i$, where $U_i$ is a neighborhood around $x$ in $X_{v_i}$. Now, by choosing a neighborhood $V_i$ of $y$ in $X_{v_i}$ such that $V_i\subseteq U_i$, we see that there exists a neighborhood $W_m(y)$ such that $W_m(y)\subseteq W_n(x)$.  
	\end{proof}
	\begin{lemma}
		The family of sets $\{W_n(x): n\in \mathbb{N},x\in X\}$ form a basis for the topology $\TT$.
	\end{lemma}
	\begin{proof}
		Using the previous lemma, it remains to prove that if $W_{n_1}(x_1)$ and $W_{n_2}(x_2)$ are two neighborhoods and $y\in W_{n_1}(x_1)\cap W_{n_1}(x_2)$ then there exists a neighborhood $W_n(y)$ around $y$ such that $W_n(y)\subseteq W_{n_1}(x_1)\cap W_{n_1}(x_2)$. Again, by the previous lemma, there exists $W_{m_1}(y)$ and $W_{m_1}(y)$ such that $W_{m_1}(y)\subseteq W_{n_1}(x_1)$ and $W_{m_1}(y)\subseteq W_{n_1}(x_2)$. Note that there exists $m$ such that $W_m(y)\subseteq W_{m_1}(y)\cap W_{m_1}(y)$ and hence the lemma.
	\end{proof}
	\begin{lemma}
		For each vertex $v$, $\pi|_{X_v}:X_v\rightarrow X$ is a continuous map.  \qed
	\end{lemma}
	Proof of the above lemma is easy, so we left it for the reader.
	\begin{lemma}
		[Regularity] The space $(X,\TT)$ is regular, i.e. for all $x\in X$ and for all $W_n(x)$ there exists $W_m(x)$ such that $\overline{W_m(x)}\subseteq W_n(x)$.
	\end{lemma}
	\begin{proof}
		Let $\eta\in (\partial T)'$ and let $W_n(\eta)$ be a neighborhood of $\eta$ in $X$. Let $v$ be a vertex from where the subtree $T_n(x)$ starts. Let $e$ be the last edge on the geodesic segment $[v_0,v]$. Observe that closure of $W_n(\eta)$ contains only extra points coming from $X_e$. Then, by definition of neighborhoods, there exist $m$ sufficiently large such that $\overline{W_m(\eta)}\subseteq W_n(\eta)$.
		
		Let $x\in \Omega'$ and let $W_n(x)$ be a neighborhood around $x$ in $X$. let $D(x)=\{v_i\}_{i\in I}$ and let $U_i$ is a open neighborhood in $X_{v_i}$ containing $x$. Observe that only extra points in $\overline{W_n(x)}$ are coming from the closure of $U_i$ in $X_{v_i}$. Since each $X_{v_i}$ is regular, choose a neighborhood $V_i$ of $x$ such that $\overline{V_i}\subseteq U_i$. By using $V_i's$, we see that there exist $m$ such that $\overline{W_m(x)}\subseteq W_n(x)$. 
	\end{proof}
	{\bf Note:} Now, by previous lemmas, we see that the topology on $X$ is second countable, Hausdorff, and regular. Then by Urysohn's metrization theorem, we see
	that $X$ is metrizable space. Also, $X$ is a perfect space as every point of $X$ is a limit point. Finally, we have the following convergence criterion:
	
	{\bf Convergence criterion :} \label{convergence criterion} A sequence $(x_n)_{n\in \mathbb{N}}$ in $X$ converges to a point $X$ if and only if $\forall n, \exists m_0\in \mathbb{N}$ such
	that $\forall m>m_0, x_m \in W_n(x)$.
	
	Next, we prove that $X$ is compact. Proof of this is almost the same as \cite[Theorem 3.11]{tomar}. For the sake of completeness, we are giving the proof.
	\begin{lemma}
		[Compactness] The space $(X,\TT)$ is compact.
	\end{lemma}
	\begin{proof}
		It is sufficient to prove that $X$ is sequentially compact. Let $(x_n)_n$ be a sequence in $X$. Let us choose a vertex $v$ in $T$ and for each $n$, choose a vertex $v_n$ (if $x_n\in (\partial T)'$ then take $v_n=x_n$) in $D(x_n)$. We see that up to extraction of a sequence, Gromov inner products $\langle v_n,v_m\rangle_v$ either remain bounded or go to infinity. In the latter case, $v_n$ converges to a point $\eta\in (\partial T)'$. It may be possible that $x$ is the point that is identified with some parabolic point of edge groups. In any case, by definition of neighborhoods, we see that there is a subsequence of $x_n$ converges to $x$. In the first case up to extracting a subsequence, we assume that Gromov inner products $\langle v_n,v_m\rangle$ is equal to some constant $N$. Let $g_n$ be the geodesic from $v$ to $v_n$, then there exist geodesic $g = [v,v']$ of length $N$ such that $g$ lies in each $g_n$. For different $n$ and $m$, $g_n$ and $g_m$ do not have a common prefix whose length is longer than $g$. We have two cases: (1) There exists s subsequence $(g_{n_k})$ of $(g_n)$ such that $g_{n_k}=g$. Then, in this case, as $X_{v'}$ is compact; there exists a subsequence of $(x_n)$ converging to a point of $X_{v'}$. (2) There exists a subsequence $(g_{n_k})$ of $(g_n)$ such that each $g_{n_k}$ is strictly longer than $g$. Let $e_{n_k}$ be the edges just after $v'$; they all are distinct. As for each edge $e$, $X_e$ is a singleton, so $(X_{e_{n_k}})$ forms a sequence of points in $X_{v'}$. Since $X_{v'}$ is compact, there exists a subsequence that converges to a point in $X_{v'}$. Then, by convergence criterion, we see that there exists a subsequence of $(x_n)_n$ converging to this point of $X_{v'}$. (It may be possible that all the $X_{n_k}$ are equal to $x$ for some $x$ in $X_{v'}$, but the sequence also converges to $x$ in this situation.)
	\end{proof}
	\section{Proof of Theorem \ref{combination theorem}}\label{4}
	First of all, we prove that the natural action of $G$ on $X$ is convergence. We prove it with the help of the following lemmata.
	\begin{lemma}
		[Large translation] Let $(g_n)_{n\in \mathbb{N}}$ be a sequence in $G$. Assume that, for some (hence any) vertex $v_0\in T$, $dist(v_0,g_nv_0) \rightarrow \infty$. Then, there is a subsequence $(g_{n_k})_{k\in \mathbb{N}}$, two points $x\in X$ and $\eta\in (\partial T)'$ such that for all compact $K\subset (X\setminus \{\eta\})$, $g_{n_k} K$ converges uniformly to $x$. 
	\end{lemma}
	\begin{proof}
		Let $x_0$ be an element in $X_{v_0}$. Since $X$ is sequentially compact, there exists a subsequence $(g_{n_k}x_0)_k$ of $(g_{n}x_{0})$ converging to a point $x\in X$. We still have that $dist(v_0,g_{n_k}v_0)\rightarrow \infty$. Note that the lenght of geodesic segment $[g_nv_0,g_nv_1]$ is equal to the length of geodesic segment $[v_0,v_1]$ for any vertex $v_1$ of $T$. Thus, for any $m$, there exists $n_m$ such that for all $k>n_m$, the geodesic segments $[v_0,g_{n_k}v_0]$,$[v_0,g_{n_k}v_1]$ have a common prefix of length more than $m$. Then, by convergence criterion \ref{convergence criterion}, for all $v\in V(T)$, $g_{n_k}x_0$ converges to $x$. Let $\eta_1,\eta_2\in (\partial T)'$. Since the triangles in $T\cup \partial T$ are degenerate. The triangle with vertices $v_0,\eta_1$ and $\eta_2$ has the center, a vertex $v$ in $T$. Then, for any $m$, the geodesic segments $[v_0,g_{n_k}v_0],[v_0,g_{_k}v]$ have a common prefix of length greater than $m$ for all sufficiently large $k$. Now, for at least one $\eta_i$ ($i=1,2$), the segments $[v_0,g_{n_k}v_0]$ and $[v_0,g_{n_k}\eta_i]$ have a common prefix of length at least $m$ for large $k$. Then, by convergence criterion, $g_{n_k}\eta_i$ converges to $x$ for at least one $y_i$. Then, there exists at most one $\eta\in (\partial T)'$ such that for all $y'\in X\setminus \{\eta\}$, $g_{n_k}y'$ converges to $x$.
		
		Let $K$ be a compact subset of $X\setminus\{\eta\}$. Then there exists a vertex $v_0$ (base vertex), a point $\xi\in (\partial T)'$ such that $W_n(\xi)$ contains $K$ and $\eta\notin W_n(\xi)$. Let $v$ be a vertex on the geodesic ray $[v_0,\xi)$ at distance $n$. By the discussion at the beginning of the proof, the sequence $g_{n_k}X_v$ converges uniformly to $x$, and the sequence $g_{n_k}W_n(\xi)$ converges uniformly to $x$. Hence the convergence is uniform on $K$. 
	\end{proof}
	\begin{lemma}
		[Small translation] Let $(g_n)_{n\in \mathbb{N}}$ be a sequence of distinct elements in $G$ such that for some (hence any) $v_0$, $dist(v_0,g_nv_0)$ is bounded. Then, there exists a subsequence $(g_{n_k})_k$, a vertex $v$, an element $x\in X_v$, and another point $y\in \Omega'$ such that for all compact $K\subset X\setminus\{y\}$, $g_{n_k}K$ converges uniformly to $x$.
	\end{lemma}
	\begin{proof}
		There are two cases to consider:
		
		Case(1) Assume that for some vertex $v$, there exists a sequence $(h_n)_n\subset G_v$ and an element $g\in G$ such that $g_n=h_ng$. Since $G_v$ acts as a convergence group on $X_v$, there exists a subsequence $(g_{n_k})_k$ of $(g_n)_n$ and a point $y\in X_{g^{-1}v}$ such that for all compact $K\subseteq X_{g^{-1}v}\setminus\{y\}$ such that $g_{n_k}K$ converges uniformly to a point $x\in X_v$. Suppose $y$ does not belong to any edge boundary adjacent to $X_{g^{-1}v}$. Then for any $z\in X\setminus \{X_{g^{-1}v}\}$ there is a geodesic segment from $g^{-1}v$ to $w$, where $z\in X_w$. This geodesic segment contains an edge adjacent to $g^{-1}v$. Then $g_{n_k}X_{e_n}$ converges uniformly to $x$, and hence $g_{n_k}z$ converges to $x$. Same is true for points in $(\partial T)'$. Now, suppose that $y$ is in $X_{e_n}$ for some $e_n$ adjacent to $g^{-1}v$. If $X_{e_n}$ is a singleton, then $y$ is a parabolic point. Therefore $g_{n_k}y$ also converges to $x$ otherwise $gy$ becomes a conical limit point which is not possible. Using the same argument as above, we see that for all $z\in X$, $g_{n_k}z$ converges to $x$. If $X_{e_n}$ has at least $2$-elements and $y\in X_{e_n}$ is not parabolic, then again, by the same argument, for all $z\in X\setminus \{y\}$, $g_{n_k}z$ converges to $x$.
		
		Case(2) Suppose such a sequence $(h_n)_n$ and vertex $v$ do not exist. Then, up to extracting a subsequence, we can assume that $dist(v_0,g_nv_0)$ is constant. Let us choose a vertex $v$ and a subsequence $(g_{n_k})_k$ such that the geodesic segments $[v_0,g_{n_k}v_0]$ share a common segment $[v_0,v]$, and the edges $e_{n_k}$ just after $v$ all are distinct. Note that the sequence $(X_{n_k})_k$ converges to a point $x$ (say). Clearly, $g_{n_k}X_{v_0}$ converges uniformly to $x$. Let $\eta\in (\partial T)'$ then the geodesics rays $[g_{n_k},\eta)$ do not pass through $v$ for sufficiently large $k$ for if $v$ is there, then, for infinitely many $k$, $g_{n_k}^{-1}v=w$ for some $w$ on the geodesic ray $[v_0,\eta)$ and we are in the first case, a contradiction. Then, it is clear that $g_{n_k}\eta$ converges to $x$. Let $z\in X\setminus\{X_{v_0}\}$. Again, we see that the geodesic segments $[g^{-1}v_0,z]$ do not pass through the vertex $v$. Thus, by definition of neighborhoods, we see that $g_{n_k}z$ converges to $x$. Hence, for all compact $K\subset X\setminus\{x\}$, we see that $g_{n_k}K$ converges uniformly to $x$. 
	\end{proof}
	Using the above two lemmata, we have the following:
	\begin{cor}
		The action of $G$ on $X$ is convergence.  \qed
	\end{cor}
	By \cite[Lemma 6.1]{tomar}, we see that every point of $(\partial T)'$ is conical for the action of $G$ on $X$. Also, by \cite[lemma 6.2]{tomar}, every point in $\Omega'$ which is the image of a conical limit point in some vertex stabilizer's boundary is conical for $G\curvearrowright X$.
	
	To prove that the action of $G$ on $X$ is geometrically finite, it remains to show that every point in $\Omega'$, which is the image of a bounded parabolic point in some vertex stabilizer's boundary, is bounded parabolic for $G\curvearrowright X$. It follows from the following lemma:
	\begin{lemma}
		Every point in $\Omega'$, which is the image of a bounded parabolic point in some vertex stabilizer's boundary, is bounded parabolic for the action of $G$ on $X$.
	\end{lemma}
	\begin{proof}
		It is sufficient to consider the amalgamated free product and HNN extension case:
		
		{\bf Amalgam case:} Let $p$ be a bounded parabolic point for a vertex group $G_v$ in $X_v$, which is not in any edge boundary attached to $X_v$. We denote $\pi(p)$ by $p$. Let $D(p) = \{v\}$. Let $P$ be the stabilizer of $p$ in $G$. Since $P$ fixes the vertex $v$, $P \leq G_v$. Note that $P$ is the stabilizer of $p$ in $G_v$. As $p$ is bounded parabolic point for $G_v$ in $X_v$, $P$ acts co-compactly on $X_v\setminus\{p\}$. Let $K$ be a compact subset of $X_v\setminus\{p\}$ such that $PK = X_v\setminus\{p\}$. Consider $\EE$, the set of edges whose boundaries intersect $K$. Let $e$ be the edge with one vertex $v$. Then there exists $h\in P$ such that $X_e \cap hK \ne \emptyset$. Therefore, the set of edges $\cup_{h\in P} h\EE $ contains every edge with one and only one vertex $v$. Let $\VV$ be the set of vertices $w$ of the tree $T$ such that the first edge of $[v,w]$ is in $\EE$. Let $K'$ be the subset of $X$ consisting of points whose domains are in $\overline{\VV}$. Define $K'' = K\cup K'$. As the sequence of points in $K'$ has the limit in $K$, $K''$ is a compact subset of $X$. It is clear that $PK'' = X\setminus \{p\}$. Thus, $p$ is a bounded parabolic point for $G\curvearrowright X$.
		
		Case(2) Suppose $p$ is an element of an edge boundary. In this case, $D(p)$ is either infinite or singleton (it is singleton when stabilizers of $p$ in adjacent vertex groups are in the edge group). In any case, proof will be the same. Suppose that $v_1$ and $v_2$ are vertices of an edge $e$. Let $X_e = \{p\}$ and $P_1,P_2$ be maximal parabolic subgroups in vertex groups $G_{v_1},G_{v_2}$ respectively, and $P$ is maximal parabolic in edge group $G_e$. Then, the domain of $p$, $D(p)$ is nothing but the Bass-Serre tree of the amalgam $Q=P_1\ast_P P_2$, which is the stabilizer of $p$ in $G$. Under the action of $Q$, the quotient of $D(p)$ is the edge $e$. Since $P_1,P_2$ acts co-compactly on $X_{v_1}\setminus\{p\},X_{v_2}\setminus\{p\}$ respectively, there exists a compact subset $K_i$ of $X_{v_i}\setminus\{p\}$ such that $P_iK_{i} = X_{v_i}$ for $i=1,2$. Consider $\EE_i$, the edges starting at $v_i$ whose boundary intersects $K_i$ but does not contain $p$. Let $e$ be an edge with only one vertex in $D(p)$ and $v_i$ be this vertex. Then there exists $h\in P_i$ such that $X_e\cap hK_i \ne \emptyset$ for $i=1,2$. Therefore the set of edges $\cup_{i=1,2} Q\EE_i$ contains every edge with one and only one vertex in $D(p)$. Let $\VV_i$ be the set of vertices $w$ such that the first edge of $[v_i,w]$ is in $\EE_i$. Let $K_i'$ be the subset of $X$ consisting of the points whose domain is in $\VV_i$ for $i=1,2$. Since a sequence of points in the spaces corresponding to the stabilizers of distinct edges in $\EE_i$ have all accumulation points in $K_i$, the set $K_i''= K_i\cup K_i'$ is a compact for $i=1,2$. Hence $K=\cup_{i=1,2}K_i''$ is a compact set of $X$ not containing $p$ and $QK =X\setminus\{p\}$. Therefore $p$ is a bounded parabolic point for $G$ in $X$.
		
		{\bf HNN extension:} Case(1) Let $G=G_{v}{\ast}_{A\simeq B}$ where $B$ is isomorphic to a subgroup of $A$. Let $p$ be a bounded parabolic point in edge group and let $P,P_A$ and $P_B$ are maximal parabolic subgroups in $G_v,A$, and $B$, respectively. Then, proof remains the same as in the amalgam case except that the maximal parabolic subgroup corresponding to edge boundary point $p$ is $P\ast_{P_A\simeq P_B}$, and maximal parabolic subgroups corresponding to parabolic points which are not in any edge spaces are maximal parabolic for $G$ in $X$.
		
		Case(2) Let $G=G_{v}{\ast}_{A\simeq B}$ where $B$ is isomorphic to $A$ (not a subgroup of $A$). Then, we can write $G$ as $G=(G_v\ast_A B)\ast_B$. Thus, by applying the amalgam and case(1) of HNN extension case respectively, we are done. 
	\end{proof}
	Hence, we have completed the proof of Theorem \ref{combination theorem}. Also, from the construction of Bowditch boundary of $G$, it is clear that vertex groups of $G(\YY)$ are relatively quasi-convex in $G$.
	\section{Proof of Theorem \ref{maintheorem}}\label{3}
	In this section, we give proof of Theorem \ref{maintheorem}. First of all, we observe the following:
	\begin{lemma}\label{rel height rel acyl}
		Let $G$ be the group satisfying the hypotheses of Theorem \ref{maintheorem}. Assume that all edge groups have finite relative height in $G$. Then, the action of $G$ on the Bass-Serre tree of $G(\YY)$ is relatively acylindrical.
	\end{lemma}
	\begin{proof}
		Since $G(\YY)$ is a finite graph of groups, the relative height of all edge groups is some fixed number $l$ (say). Then, for sufficiently large $k$, the stabilizer of length $k$ geodesic segment contains more than $l$ conjugates of some fixed edge stabilizer. Hence, by definition of relative height, we see that the action of $G$ on the Bass-Serre tree of $G(\YY)$ is relatively acylindrical.
	\end{proof}
	For the convenience of the reader, we recall Theorem \ref{maintheorem} here.
	\begin{theorem}
		Let $G(\YY)$ be a finite graph of groups satisfying the following:
		\begin{enumerate}
			\item vertex groups are relatively hyperbolic, and edge groups are relatively quasi-convex in adjacent vertex groups,
			\item the induced parabolic structures on each edge group from adjacent vertex groups are the same,
			\item $G=\pi_1(G(\YY))$ is relatively hyperbolic with respect to subgroups corresponding to natural cone-locus,
			\item For $n\geq k$, $\bigcap\limits_{i=1}^{n}\Lambda(G_{e_i})=\emptyset$ if $\bigcap\limits_{i=1}^{n} G_{e_i}$ is finite and $\bigcap\limits_{i=1}^{n}\Lambda(G_{e_i})$ is a singleton if $\bigcap\limits_{i=1}^{n} G_{e_i}$ is infinite, where $k$ is some fixed natural number.
		\end{enumerate}
		Then, edge groups have finite relative height in $G$ if and only if they are relatively quasi-convex in $G$.
	\end{theorem}
	\begin{proof}
		If edge groups are quasi-convex, then, by \cite{hruska-wise}, they have finite relative height.
		Conversely, assume that edge groups have a finite relative height. Then, by Lemma \ref{rel height rel acyl}, we see that the action of $G$ on the Bass-Serre tree of $G(\YY)$ is relatively acylindrical. In condition(4), we take $k$ as relative acylindricity constant. Now, $G(\YY)$ satisfies the hypotheses of Theorem \ref{combination theorem}. Thus, $G$ is relatively hyperbolic with respect to the same parabolic structure as in our assumption. Also, we have explicit construction of Bowditch boundary $X$ of $G$. From the construction of $X$, we see that all vertex groups are relatively quasi-convex. Thus, all edge groups are relatively quasi-convex by \cite[Lemma 2.3]{bigdely-wise}. Hence the theorem.
	\end{proof}
	\section{Application}
	In \cite{swarup-weakhyp}, Swarup proved that if a hyperbolic group $G$ splits as $A\ast_C B$ or $A\ast_C$, where $A,B,C$ are quasi-convex in $G$ then a finitely generated subgroup $H$ of $G$ is quasi-convex if and only if the intersection of $H$ with conjugates of $A,B,C$ in $G$ is quasi-convex in $G$ \cite[Theorem 2]{swarup-weakhyp}. In \cite[Theorem 5.4]{pal-rel-height}, Pal generalized the theorem of Swarup in the setting of relatively hyperbolic groups with some restrictions on parabolics. This section provides a proof same result in a more general setting.
	\begin{theorem}\label{qc}
		Let $G(\YY)$ be a graph of groups satisfying the hypotheses of Theorem \ref{combination theorem} and let $H$ be a finitely generated subgroup of $G$. Then, $H$ is relatively quasi-convex in $G$ if and only if the intersection of $H$ with conjugates of vertex groups in $G$ is relatively quasi-convex in vertex groups.
	\end{theorem}
	\begin{proof}
		If $H$ is relatively quasi-convex in $G$, then the conclusion follows as vertex groups of $G(\YY)$ are relatively quasi-convex in $G$ (see \cite[Lemma 2.3]{bigdely-wise}). Conversely, assume that the intersection of $H$ with conjugates of vertex groups in $G$ is relatively quasi-convex in vertex groups. Since $H$ is a finitely generated subgroup, $H$ admits a finite subgraph of groups $G(\YY_1)$ structure whose vertex groups are the intersection of $H$ with conjugates of vertex groups of $G(\YY)$ and edges groups are the intersection of $H$ with conjugates of edge groups of $G(\YY)$. Then, by our assumption, $G(\YY_1)$ is a graph of relatively hyperbolic groups with relatively quasi-convex edge groups. Note that the action of $H=\pi_1(G(\YY_1))$ on Bass-Serre tree $T_1$ of $G(\YY_1)$ is relatively acylindrical. Thus, $G(\YY_1)$ satisfies all the conditions of Theorem \ref{combination theorem}. Therefore, $H$ is a relatively hyperbolic group with explicit construction of Bowditch boundary, say $X_1$. Also, from the construction of boundaries, $X_1$ is a minimal closed $H$-invariant subset of $X$. Thus, $X_1$ is the limit set of $H$ in $G$. Since $H$ acts geometrically finitely on $X_1$, $H$ is a relatively quasi-convex subgroup of $G$.
	\end{proof}
	A relatively hyperbolic group is said to be locally quasi-convex if every finitely generated subgroup is relatively quasi-convex. Let $G(\YY)$ be a graph of groups as in Theorem \ref{maintheorem}. Additionally, assume that each vertex group of $G(\YY)$ is locally quasi-convex. Let $H$ be any finitely generated subgroup of $G$. Assume that the intersection of $H$ with conjugates of edge groups is finitely generated in edge groups. Then, we have the following:
	\begin{prop}
		$G$ is a locally quasi-convex relatively hyperbolic group.
	\end{prop}
	\begin{proof}
		Let $H$ be a finitely generated subgroup of $G$. Then, $H$ admits a finite subgraph of groups $G(\YY_1)$ structure with finitely generated edge groups. Then, by \cite[Lemma 2.5]{bigdely-wise}, vertex groups are also finitely generated. As vertex groups of $G(\YY)$ are locally quasi-convex, all vertex groups of $G(\YY_1)$ are relatively quasi-convex. Then, edge groups of $G(\YY_1)$ are also relatively quasi-convex in corresponding edge groups of $G(\YY)$. Applying Theorem \ref{qc}, we see that $H$ is relatively quasi-convex in $G$.
	\end{proof}

\end{document}